\documentclass[10pt]{article}
\usepackage{amsfonts,amsmath,amssymb,amsthm}
\usepackage{hyperref}
\usepackage{graphicx}
\usepackage[affil-it]{authblk}

\newcommand{\detail}[1]{\par\noi{\bf [Proof detail\ }{#1}
\hfill{\bf ]}\par\noi\hspace{-4pt}}
\renewcommand{\detail}[1]{}

\newcommand{\noi}{\noindent}

\newtheorem{theorem}{Theorem}[section]
\newtheorem{proposition}[theorem]{Proposition}
\newtheorem{corollary}[theorem]{Corollary}
\newtheorem{conjecture}[theorem]{Conjecture}
\newtheorem{lemma}[theorem]{Lemma}

\newtheorem{remark}[theorem]{Remark}

\newcommand{\et}{\end{theorem}}
\newcommand{\bl}{\begin{lemma}}
\newcommand{\el}{\end{lemma}}
\newcommand{\bp}{\begin{proposition}}
\newcommand{\ep}{\end{proposition}}
\newcommand{\bcor}{\begin{corollary}}
\newcommand{\ecor}{\end{corollary}}
\newcommand{\br}{\begin{remark}\rm}
\newcommand{\er}{\end{remark}}
\newcommand{\bcon}{\begin{conjecture}}
\newcommand{\econ}{\end{conjecture}}

\newcommand{\be}{\begin{equation}}
\newcommand{\ee}{\end{equation}}
\newcommand{\ba}{\begin{array}}
\newcommand{\ea}{\end{array}}
\newcommand{\bc}{\be\begin{array}{r@{\,}c@{\,}l}}
\newcommand{\ec}{\end{array}\ee}

\newcommand{\Ga}{\Gamma}

\newcommand{\eps}{\varepsilon}

\newcommand{\La}{\Lambda}
\newcommand{\sig}{\sigma}

\newcommand{\R}{{\mathbb R}}
\newcommand{\N}{{\mathbb N}}
\newcommand{\Z}{{\mathbb Z}}

\renewcommand{\P}{{\mathbb P}}

\newcommand\TT{{\cal T}}

\newcommand{\uaw}{\uparrow}

\newcommand{\daw}{\downarrow}
\newcommand{\ben}{\begin{enumerate}}
\newcommand{\een}{\end{enumerate}}
\newcommand{\ii}{\item}
\newcommand{\beqn}{\begin{eqnarray}}
\newcommand{\eeqn}{\end{eqnarray}}
\newcommand{\beqnn}{\begin{eqnarray*}}
\newcommand{\eeqnn}{\end{eqnarray*}}

\newcommand\Tr{\mbox{Tr}^h}
\newcommand\cT{{\cal T}}

\title{Trimming a Tree and the Two-Sided  Skorohod Reflection.}
\date{\today}

\author[1,2]{Emmanuel Schertzer}
\affil[1]{UPMC Univ. Paris 06,
Laboratoire de Probabilit\'es et Mod\`eles Al\'eatoires, CNRS UMR 7599, Paris, France.}
\affil[2]{Coll\`ege de France,
Center for Interdisciplinary Research in Biology, CNRS UMR 7241, Paris, France.}   

\begin{document}

\maketitle


\begin{abstract}
The $h$-trimming of a tree is a natural regularization procedure
which consists in
pruning the 
small branches of a tree: given $h\geq0$,
it is obtained by only keeping the vertices 
having at least one leaf above them at a distance greater or equal to $h$.

The $h$-cut of a function $f$ is the function
of minimal total variation uniformly approximating the increments of $f$ with accuracy $h$, and
can be explicitly constructed via 
the two-sided Skorohod reflection of $f$ on the interval $[0,h]$. 

In this work, we show that the contour path of the $h$-trimming
of a rooted real tree is given by the $h$-cut of its original 
contour path. We provide two applications of this result. First,
we recover a famous result of 
Neveu and Pitman \cite{NP89}, which states 
that the $h$-trimming of a tree coded by a Brownian excursion  
is distributed as a standard binary tree. In addition,
we provide the joint distribution of this Brownian tree and its trimmed version
in terms of the local time of the two-sided reflection of its contour path. 
As a second application,
we relate the maximum of a sticky Brownian motion
to the local time of its driving process.
\end{abstract}

\section{Introduction and Main Results.}
\label{Intro}

In a rooted tree, there is a natural partial ordering
on the set of vertices -- $x \preceq y$ iff the unique path
from the root to vertex $y$ passes through vertex $x$. 
Under this ordering, the children of a given node are not ordered. 
However, one can always specify some arbitrary
ordering of the children of each vertex of the tree (from left to right) 
and by doing so,
one defines an object called a rooted {\it plane} tree -- see Le Gall \cite{LG05} for a 
formal definition.

Every rooted plane tree can be encoded
by its contour path, where the contour path 
can be loosely understood
by envisioning the tree as embedded in
the plane, 
with each of its edges 
having unit length.
We can then imagine  
a particle starting
from the root,
traveling
along the edges of the tree
at speed $1$ and exploring 
the tree from left to right --- see Fig \ref{fig1}. 
The contour path of the tree
is simply defined as the current distance
of the exploration particle to the root --- see Fig \ref{fig2}. 

In this paper, we show that the contour path of the 
$h$-trimming of a rooted plane tree 
(and more generally the $h$-trimming of rooted real trees)
is given by the $h$-cut of the 
original contour path; where the $h$-cut is constructed from the two-sided Skorohod reflection
of the original contour path -- see (\ref{h-truncation}).


\bigskip

{\bf Real rooted  trees.} As already discussed, every rooted plane tree can be encoded by
its contour path which is a function in $C_0^+(\R^+)$
-- the set  of continuous non-negative functions on $\R^+$ with $f(0)=0$ and compact support.
Conversely, it is now well established that any  $f\in C_0^+(\R^+)$
encodes a {\it real rooted tree} in the following natural way -- see again \cite{LG05}
for more details. 
Define 
$$
\forall s,t \in \R^+, \ \ d_f(s,t) = f(s) + f(t) - 2 \inf_{[s\wedge t,s\vee t]} f, 
$$
and the equivalence relation $\sim$ on $\R^+$ as follows
$$
s \sim t \Longleftrightarrow d_f(s,t) = 0.
$$
The equivalence relation $\sim$ defines a quotient space
$$
{\cal T}_f \ = \ \R^+ / \sim
$$
referred to as the tree encoded by $f$. 
The function $d_f$ induces a distance on ${\cal T}_f$, and we keep the notation 
$d_f$ for this distance. In \cite{LG05}, it is shown that the pair $({\cal T}_f,d_f)$
defines a real tree in the sense 
that the two following properties are satisfied. For every $a,b\in {\cal T}_f$:
\begin{enumerate}
\item[(i)]{\bf (Unique geodesics.)}  There is a unique isometric map $\psi^{a,b}$ from $[0,d_f(a,b)]$ into ${\cal T}_f$ such that $\psi^{a,b}(0) = a$ and $\psi^{a,b}({d_f(a, b)}) = b$.
\item[(ii)]{\bf (Loop free.)}   If $q$ is a continuous injective map from $[0, 1]$ into ${\cal T}_f$ , such that $q(0) = a$ and $q(1) = b$, we have
$q({[0, 1]}) = \psi^{a,b}({[0, d_f(a, b)]})$.
\end{enumerate}
In the following, for any $x,y\in{\cal T}_f$, $[x,y]$
will denote the geodesic from $x$ to $y$, i.e., $[x,y]$
is the image of $[0,d_f(x,y)]$ by $\psi^{x,y}$. We will denote by $p_f$ the canonical projection from ${\R}^+$
to ${\cal T}_f$ which can be thought of
as the position of the exploration particle
at time $t$. In the following, 
$\rho_f=p_f(0)$ will be referred to as the root of the tree ${\cal T}_f$.
In what follows, real trees will always be rooted, even if this is not mentioned explicitly.

$d_f$ induces a natural partial ordering on the rooted tree ${\cal T}_f$ :  
$v' \preceq v$ ($v'$ is an ancestor of $v$)  iff 
$$
d_f(v,v') \ = \ d_f(\rho_f,v) - d_f(\rho_f,v').  
$$
We note that this partial ordering is directly related to the sub-excursions nested in the function $f$. Indeed,
for any $s,t\geq0$,
$p_f(t) \preceq p_f(s)$ if and only if $\inf_{[t\wedge s, t\vee s]} f = f(t)$, which is equivalent
to saying that $t$ is the ending time or starting time of a sub-excursion of $f$ starting from level $f(t)$
and straddling time $s$ -- see Fig \ref{fig1} and \ref{fig2}.

Finally,
for any $x,y\in \cT_f$,
 the most recent common ancestor of $x$ and $y$ -- denoted by  $x\wedge y$ -- 
is defined as $\sup\{ z\in \cT_f : z \preceq x,y \}$. From the definition of our genealogy,
for any $t_1,t_2\in\R^+$, we must have
\be\label{ionf}
p_f(t_1)\wedge p_f(t_2) =  p_f(s), \ \ \mbox{for any $s\in\mbox{argmin}_{[t_1\wedge t_2,t_1 \vee t_2]} f$,}
\ee
with the height of the most recent common ancestor being given by $f(s)=\min_{[t_1\wedge t_2, t_1\vee t_2]} f$. 

\bigskip

\begin{figure}[ht]
\begin{minipage}[b]{0.5\linewidth}
\centering
\includegraphics[scale=0.2]{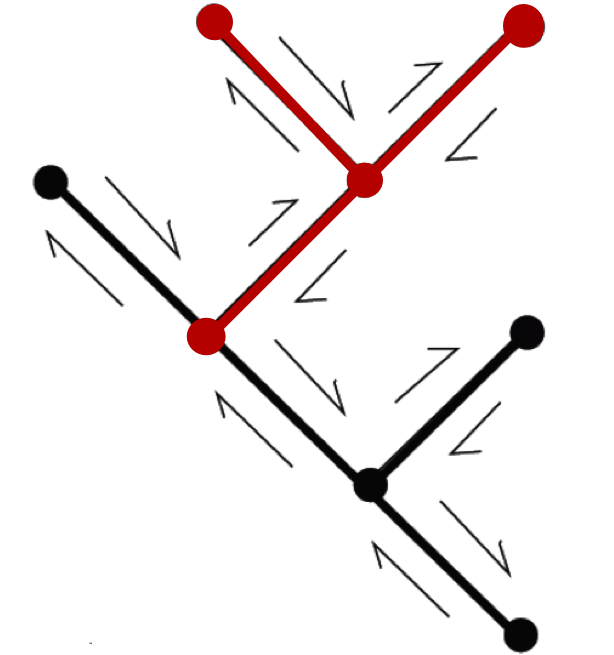}
\caption{{\it Exploration of a plane tree. The exploration particle travels 
along each branch twice : first on the left and away from the root, and then on the right
and towards the root. The root of the red sub-tree belongs to the $2$-trimming of the tree.}}
\label{fig1}
\end{minipage}
\hspace{0.5cm}
\begin{minipage}[b]{0.5\linewidth}
\centering
\includegraphics[scale=0.2]{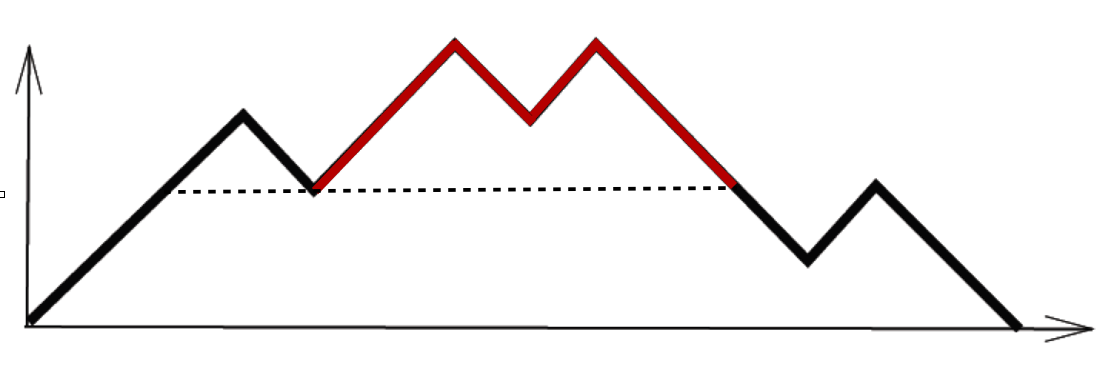}
\caption{{\it Contour path. The red portion of the curve 
is a sub-excursion of height 2 corresponding to the exploration of the red sub-tree 
on the left panel.}}
\label{fig2}
\end{minipage}
\end{figure}

\bigskip

{\bf Trimming and the two-sided Skorohod reflection.} As in Evans \cite{E05}, for every $h>0$,
$({\cal T}_f,d_f)$
as the (possibly empty) sub-tree
\be\label{def-tr}
\mbox{Tr}^h( {\cal T}_f ) := \{ x\in{\cal T}_f \ : \ \sup_{y\in {\cal T}_f \ : \ y \succeq x} d_f(x,y) \geq h \},
\ee
which consists of all the points in ${\cal T}_f$
having at least one leaf above them at distance greater or equal to $h$.  (Note that 
$\mbox{Tr}^h( {\cal T}_f )$  is not empty if and only if $\sup_{[0,\infty)} f\geq h$.) 
As already mentioned, one of the main results of this paper is the relation between
the $h$-trimming of a real rooted tree and  
the two-sided Skorohod reflection of its contour path. 
The one-sided Skorohod reflection is well known
among probabilists. Given
a continuous function $f$ starting from $x\geq0$,
it is simply defined as
the following transformation
\be
\Ga^{0}(f)(t) \ = \  f(t) \ - \ \inf_{[0,t]} f \wedge 0. \label{one-sided}
\ee
The resulting path obviously remains non-negative and 
the function
$c(t)\equiv-\inf_{[0,t]} f$ is easily
seen
to be the unique solution 
of the so-called (one-sided) Skorohod equation, i.e., $c$ is the continuous function $c$ on $\R^+$ such that $c(0)=0$
and
\begin{enumerate}
\item $\Ga^0(f)(t) := f(t) + c(t)$ is non-negative.
\item $c$ is non-decreasing.
\item $c$ does not vary off the set $\{t \ : \ \Ga^0(f)(t) = 0 \}$, i.e.,
the support
of the measure $d c$ is contained in $\Ga^{0}(f)^{-1}(\{0\})$ .
\end{enumerate}

See Lemma 6.17 in \cite{KS91} for a proof of this statement. 
Intuitively,
the solution $c$,
which will be 
referred to as the {\it compensator}
of the reflection in the rest of this paper, can be thought 
of as the minimal
amount of upward push that one needs to 
exert on the path $f$
to keep it away from negative values.  
The Skorohod equation states that the reflected path is completely driven
by $f$ when it is away from the origin, while it is repealed from negative values
by the compensator upon reaching level $0$.  
The following theorem
is a generalization of the Skorohod 
equation to the two-sided case.  

\begin{theorem}[{\bf Two-Sided Skorohod Reflection}]
\label{Skor}
Let $h\geq0$ and let $f$ be a continuous function with $f(0)\in[0,h]$. 
There exists a unique 
pair of
continuous functions $(c^0(f),c^h(f))$ with $c^0(f)(0)=c^h(f)(0)=0$ satisfying the three following properties.
\begin{enumerate}
\item $\La_{0,h}(f)(t) \ := \ f(t) + c^h(f)(t) + c^0(f)(t)$ is valued in $[0,h]$.
\item $c^0(f)$ (resp, $c^h(f)$) is a non-decreasing (resp., non-increasing) function.
\item  $c^0(f)$ (resp, $c^h(f)$) does not vary off the set $\La_{0,h}(f)^{-1}(\{0\})$ (resp., $\La_{0,h}(f)^{-1}(\{h\})$) 
\end{enumerate}
\end{theorem}

As noted by
Kruk, Lehoczky, Ramanan, and Shreve \cite{KLRS07}, existence and uniqueness to the Skorohod problem 
follow directly from Lemma 2.1, 2.3 and 2.6 in Tanaka \cite{T79}.
In the rest of this paper,
$\La_{0,h}(f)$ will be referred to as
the two-sided
Skorohod
reflection of the path $f$ on $[0,h]$,
while
the pair of functions $(c^0(f),c^h(f))$ will be referred to
as the compensators associated with the function $f$. 
In the same spirit as the one-sided 
reflection,
the compensator $c^h(f)$ (resp., $c^0(f)$) can be thought of as
 the minimal amount
of downward (resp., upward) push at level $h$ (resp., $0$) that one has to exert on $f$ to keep the path 
$\La_{0,h}(f)$ inside the interval $[0,h]$. In other words,
adding the compensators $c^0(f)$ and $c^h(f)$ to $f$
is the ``laziest way'' of keeping $f$ in the interval $[0,h]$. 


\bigskip

Let $f$ be a continuous function on $\R^+$
with $f(0)=0$ (with no restriction on the support and on the sign of $f$). 
For such a function, define the $h$-cut of the function $f$ as
\be\label{h-truncation}
f_h  \ := \ f-\Lambda_{0,h}(f) = -c^0(f) - c^h(f).
\ee
$f_h$ is
also characterized by an interesting variational property. Indeed,
combining Proposition 2 in Mil{\l}o\'s \cite{M13} and Corollary 3.12 in 
{\L}ochowski \cite{L13}, we get that
for every interval $[0,t]$,
$f_h$ is the unique solution of the minimization problem
\be\label{h-truncation2}
\mbox{argmin}_{g \ : \ g(0)=0, ||f-g||_{osc,[0,t]} < h}  \  TV(g,[0,t]),
\ee
where $||h||_{osc,[0,t]} \ = \ \sup_{x,y\in[0,t]} \ |h(x)-h(y)|$
and $TV(g,[0,t])$ is the total variation of $g$ on the interval
$[0,t]$. In other words,
$f_h$ is the function of minimal total variation 
uniformly approximating the 
increments of $f$ with accuracy $h$.
\footnote{Again following \cite{M13}, the $h$-cut of the function $f$ is a translation of the so-called 
$h$-truncation
of $f$, as introduced and studied by {\L}ochowski  \cite{L11}, \cite{L13}-- see also {\L}ochowski and  Mi{\l}o\'s \cite{LM13}.}

%
Our main  theorem states that the contour path
of the  $h$-trimming of a tree is simply given by 
the $h$-cut of its original contour path.

\begin{theorem}\label{main1}
Let $f\in C_0^+(\R^+)$ and let us assume that the $h$-trimming
of ${\cal T}_f$ is not empty.  
\begin{enumerate}
\item The $h$-cut $f_h$
belongs to $C_0^+(\R^+)$. 
\item 
The $h$-trimming of the 
real tree $({\cal T}_f,d_f)$  is identical to the real tree $({\cal T}_{f_h},d_{f_h})$
(up to a root preserving isometry).
\end{enumerate}
\end{theorem}

To state our next result, we need to introduce some extra notations.
For a continuous function $f$ with $f(0)=0$, define $t_n(f)\equiv t_n$ (resp., $T_n(f)\equiv T_n$) to be the $n^{th}$ returning time at level $0$ (resp., $h$) of $\La_{0,h}(f)$ 
and $s_n(f)\equiv s_n$ to be the $n^{th}$ exit time at $0$ of $\La_{0,h}(f)$ as follows. $t_0=0$, and for $n\geq1$
\beqn
T_n & :=  \inf\{ u> t_{n-1} & : \La_{0,h}(f)(u)= h \} \nonumber \\
t_n & :=  \inf\{ u> T_{n}    & :    \La_{0,h}(f)(u)= 0   \} \nonumber \\
s_n & := \sup\{u\in[t_{n-1},t_{n}) & : \ \La_{0,h}(f)(u)=0 \} \label{def:T}
\eeqn 
with the convention that $\sup\{\emptyset\}, \inf \{\emptyset\}=\infty$.
Let $N_h(f)$ be the number of returns of $\La_{0,h}(f)$ to $0$, i.e.
\be
N_h(f) \ = \ \sup\{n: \ t_n < \infty\}. \nonumber
\ee
Finally, define
\beqn
\forall n\geq1, & X_n(f) = &  f_{h}(t_n) - f_h(s_{n}),  \nonumber \\
		      & Y_n(f) =   &  f_h(t_{n-1}) - f_h(s_{n}). \label{xy}
\eeqn
As we shall see below (see Theorem \ref{PN}(2)), when $f$ is a Brownian excursion,
the quantity $X_n(f)$ (resp., $Y_n(f)$) simply
coincides with the amount of Brownian local time 
accumulated by the reflected path $\La^{0,h}(f)$ at $h$ (resp., $0$) on the interval $[t_{n-1},t_n]$.

\begin{proposition}\label{algo} Let $f\in C_0^+(\R^+)$
and let us assume that the $h$-trimming
of ${\cal T}_f$ is not empty. 
The $h$-trimming of ${\cal T}_f$
is equal (up to a root preserving isometry) to the tree
generated inductively according to the following  algorithm -- see Fig \ref{my-tree}. 
\begin{enumerate}
\item[(Step 1.)] Start with a single branch of length $X_1$.
\item[(Step n, $n\geq2$)] If $n=N_h(f)$ stop. Otherwise, let $z_{n-1}$ be the tip of the $(n-1)^{th}$ branch. On the ancestral line $[\rho,z_{n-1}]$, graft a branch of length $X_{n}$
at a distance $Y_{n}$ from the leaf $z_{n-1}$. 
\end{enumerate}
\end{proposition}

\begin{figure}[ht]
\centering
\includegraphics[scale=0.3]{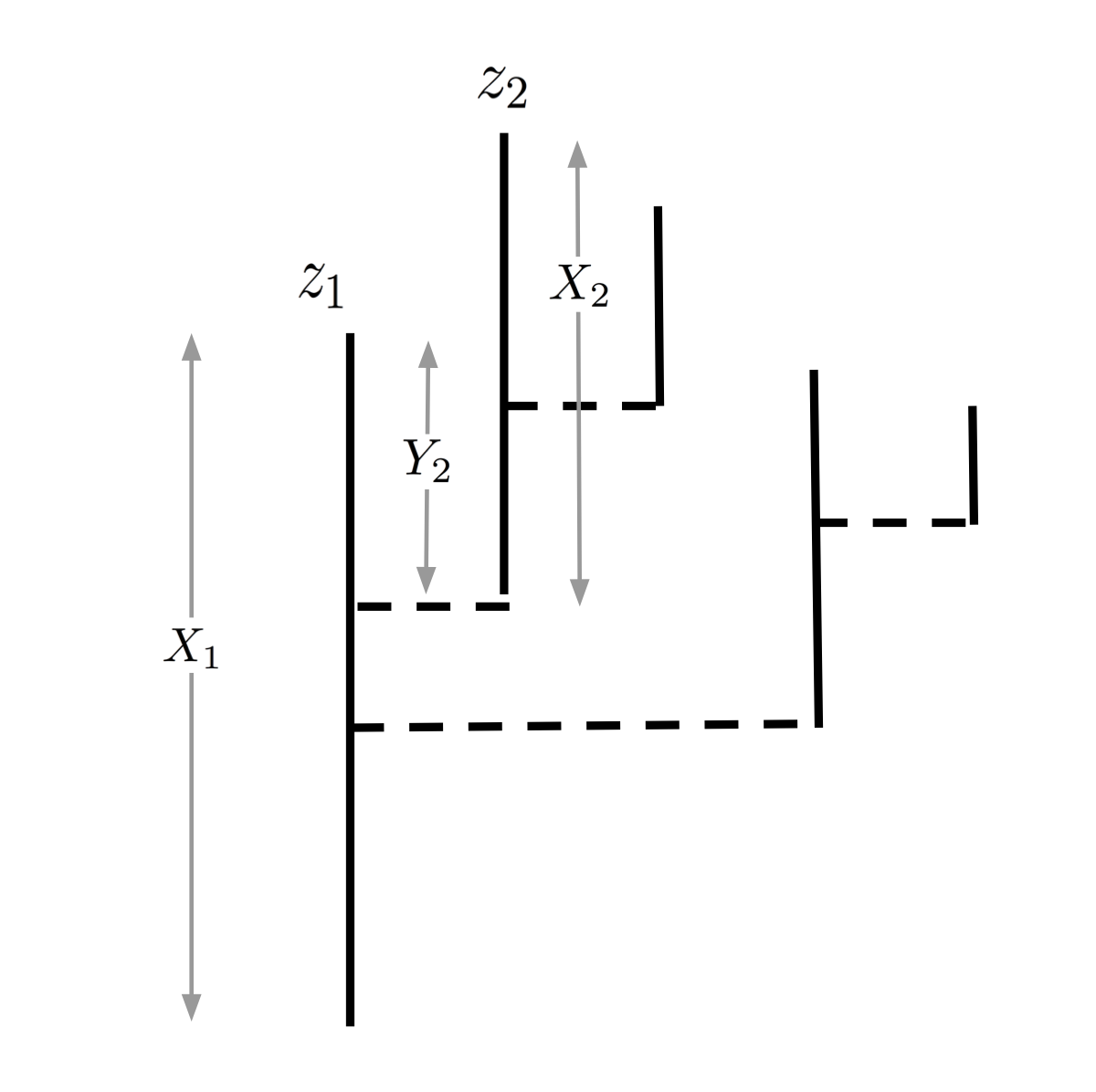}
\caption{{\it Schematic representation of the algorithm generating the $h$-trimming of a tree
from the two-sided reflection of its contour path.}}
\label{my-tree}
\end{figure}

{\bf Relation with standard binary trees.} 
Recall that  standard binary trees have branches 
(1) that have i.i.d. exponential life time with mean $\alpha$, and
(2) when they die, they either give birth to two new branches, or have no offspring with equal probability $1/2$.
The algorithm described in Proposition \ref{main1} is reminiscent of a classical construction of standard 
binary trees (see e.g., \cite{LG89}),
for which $\{(X_n(f),Y_n(f))\}$ are replaced with an infinite sequence of independent exponential r.v.'s $\tilde X_1, \tilde Y_2, \tilde X_2, \tilde Y_3,\cdots$
with parameter $\alpha$ and the algorithm stops at step $\tilde N$, with
\be
\tilde N := \inf\{n \ : \ \sum_{i=1}^{n} (\tilde X_i-\tilde Y_{i+1}) < 0 \}. 
\ee
(Note that this stopping condition is quite natural:  
the quantity 
$\sum_{i=1}^{n} (\tilde X_i-\tilde Y_{i+1})$ 
is the height of the $n^{th}$ intercalated branching point. 
We stop the algorithm once the branching point has negative height.)
Using Proposition \ref{algo},
we easily recover a result due to 
Neveu and Pitman \cite{NP89}, relating 
the $h$-trimming of the tree encoded by a Brownian excursion
with standard binary trees (see item 1. in the following theorem).
Further, the next theorem provides the joint distribution of the tree ${\cal T}_e$
and its trimmed version  $\Tr({\cal T}_e)$ (see item 2).
In the following, we 
define 
\beqn\label{local-time}
l^h(w)(t) & := &  \lim_{\eps\daw0}\frac{1}{2\eps} |\{s\in[0,t] \ : \ \Lambda_{0,h}(w)(s)\in[h-\eps,h] \}| \nonumber \\
l^0(w)(t) & := &  \lim_{\eps\daw0}\frac{1}{2\eps} |\{s\in[0,t] \ : \ \Lambda_{0,h}(w)(s)\in[0,\eps] \}|,\
\eeqn
provided that those limits exist. $l^h(w)$ (resp., $l^0(w)$) will be
referred to as the local time of $\La^{0,h}(w)$ at $h$ (resp., at $0$).

\begin{theorem}\label{PN}
Let $e$ be a Brownian excursion conditioned on having a height larger 
than $h$. 
\begin{enumerate}
\item The $h$-trimming of the tree $(\cT_e,d_e)$
is a standard binary tree with parameter $\alpha=h/2$.
\item  For $1 \leq i\leq N_h(e)$,
\begin{itemize}
\item  $X_i(e)$ a.s. coincides with the local time of $\La_{0,h}(e)$ at $h$ accumulated between $[t_{i-1}(e),t_{i}(e)]$, i.e., 
$X_i(e)=l^h(w)(t_i)-l^h(w)(t_{i-1})$. 
\item  $Y_i(e)$ a.s. coincides with the local time of $\La_{0,h}(e)$ at $0$ accumulated between $[t_{i-1}(e),t_{i}(e)]$.
\end{itemize}
\end{enumerate}
\end{theorem}




\bigskip

{\bf The maximum of a sticky Brownian motion.} 
Our final application of Theorem \ref{main1}
relates to the sticky Brownian motion.
Given a filtered probability space
$(\Omega,{\cal G}, \{{\cal G}_t\}_{t\geq0},\P)$,
a sticky Brownian motion, with parameter 
$\theta>0$,
is defined as the adapted process
taking value on $[0,\infty)$ 
solving the following stochastic
differential equation (SDE):
\beqn
\label{sticky-bm}
d z^\theta(t) \ = \ 1_{z^\theta(t)>0} d w(t) \  + \ \theta \ 1_{z^\theta(t)=0} d t, 
\eeqn
where $(w(t); \ t \geq0)$ is a standard ${\cal G}_t$-Brownian motion. 
Intuitively, $z^\theta$ is driven by $w$ away from level $0$,
and gets an upward push upon reaching this level,
keeping the process away from negative values. 
Sticky Brownian motion were first investigated by Feller \cite{F57}
on strong Markov processes taking
values in $[0,\infty)$ that behave like Brownian motion away from 0.
We refer the reader to Varadhan
\cite{V01} for a good introduction on this object.

Ikeda and Watanabe showed that (\ref{sticky-bm})
admits a unique weak solution. The result was later
straightened by Chitashvili \cite{C89} and 
Warren \cite{W99} 
who showed that
$z^\theta$ is not measurable with respect to $w$ and that, in order to construct
the process $z^\theta$, one needs to add some extra randomness
to the driving Brownian motion $w$. In \cite{W02},
Warren did exhibit this extra randomness 
and showed that it can be expressed
in terms of a certain marking procedure 
of the random tree induced by the reflection of the
driving Brownian motion $w$ (more on that in Section \ref{proof-sect}).

Among the first applications related to
sticky Brownian motions, we cite Yamada \cite{Y94} and Harrison and Lemoine \cite{HL81} 
who
studied sticky random walks as the limit of storage processes. 
More recently, Sun and Swart \cite{SS08} 
introduced a new object called the Brownian net 
which can be thought of as an infinite 
family of one-dimensional coalescing branching Brownian motions
and in which sticky Brownian motions play an essential role (see also Newman, Ravishankar and Schertzer \cite{NRS10}). 

Building on the approach of Warren \cite{W02},
and using Theorem \ref{main1}, we will show that
the law of the maximum of a sticky Brownian motion can be expressed in
terms of the local time of the two-sided reflection of its driving Brownian motion  $w$ on the interval $[0,h]$.

In the following,
$\lambda_{0,h}(\cdot)$ will refer to the linear function reflected at $0$ and $h$, i.e., the function obtained by a linear interpolation of
the points 
$\{(2n \cdot h,0)\}_{n\in\Z}$ and $\{((2n+1) \cdot h),h\}_{n\in\Z}$. For any 
continuous function $f$,
the {\it standard} reflection of $f$ on $[0,h]$ (as opposed to the two-sided Skorohod reflection) will refer to the transformation $\lambda_{0,h}(f)$. In \cite{W99}, the one-dimensional 
distribution
of a sticky Brownian motion 
conditionally on its driving process was given. 
The following theorem
provides the one-dimensional distribution 
of the maximum of a sticky Brownian motion
conditionally on its driving process.

\begin{theorem}
\label{teo1}
Let $h>0$ and let $(z^\theta(t),w(t); t\geq0)$ be a weak solution of equation (\ref{sticky-bm})
starting at $(0,0)$. 
$\Lambda_{0,h}(w)$ is distributed as a Brownian motion reflected (in the standard way) on  $[0,h]$. 
Furthermore,
\beqnn
\P\left( \max_{[0,t]} z^\theta\ \leq \ h \ | \ \sigma(w) \right) 
\ = \ \exp\left( -2\theta \ l^h(w)(t) \right), \ \ 
\eeqnn
where $l^h(w)$ is the local time at $h$ for the path $\Lambda_{0,h}(w)$ (see (\ref{local-time})).
\end{theorem}

\section{Proof of Theorem \ref{main1} and Proposition \ref{algo}}

In the following, a nested sub-excursion of the function $f\in C_0^+(\R^+)$
will refer to any section of the path $f$ on an interval $[t_-,t_+]$
such that 
$\forall t\in[t_-,t_+], \inf_{[t_-,t]} f = f(t-)=f(t_+)$ -- see Fig \ref{fig2}. The height of such a sub-excursion is defined as $\max_{[t_-,t_+]} f - f(t_+)$.

Let $p_f$ be the canonical projection from ${\R}^+$
to ${\cal T}_f$, which can be thought of as
the position
of the exploration particle at time $t$.
By definition, $p_f(t)$ belongs to $\mbox{Tr}^h( {\cal T}_f )$ if and only if there exists 
$s$ such that $\inf_{[t\wedge s, t\vee s]} f = f(t)$ and $f(s)-f(t)\geq h$, which is equivalent
to saying that $t$ is the ending time or starting time of a sub-excursion (nested in $f$) of height at least $h$ starting from level $f(t)$. 
We claim that the extreme points
of $\Tr(\cT_f)$ -- or {\it leaves} -- are contained in the set of points of the form $z=p_f(t)$, where $t$
is the time extremity of a sub-excursion of height {\it exactly} $h$. 
In order to see that, 
let $t$ be the time extremity of a sub-excursion of height strictly larger   
than $h$.  By continuity, this sub-excursion must contain a sub-excursion
of height exactly $h$. Thus, $p_f(t)$ must have at least one  descendant and can not be a leaf. 
As claimed earlier, this 
shows that the leaves must be visited at 
the time extremities of some sub-excursion of height {\it exactly} $h$. 
\footnote{
Note that we only have an inclusion. For example, let $t$ be the starting time 
of a sub-excursion of height exactly $h$,
and let $t_e$ be the ending time of this excursion. It $t_e$ 
is the starting time of another sub-excursion of size $>h$, then
$p_f(t)$ is not a leaf of the $h$-trimming of the tree.}

Let us now define 
inductively $\{\tau_n(f)\}_{n\geq 0}\equiv\{\tau_n\}_{n\geq 0}$, 
$\{\theta_{n}(f)\}_{n\geq1}\equiv \{\theta_{n}\}_{n\geq1}$ and $\{\sigma_{n}(f)\}_{n\geq0}\equiv \{\sigma_{n}\}_{n\geq0}$
as follows : $\tau_0=0$, $\sigma_0=0$ and
\beqn
\theta_{n+1} \ = \ \inf\{s > \tau_n: \   f(s) -h =  \inf_{[\tau_n,s]} f\}.\nonumber \\
\tau_{n+1} \ = \ \inf\{s > \theta_{n+1} : \sup_{[\theta_{n+1},s]} f = f(t)+h\}, \nonumber \\
\sig_{n+1} \ = \ \sup\{s\in[\tau_n,\tau_{n+1}] :  f(s) =  \inf_{[\tau_n,s]} f\}. \label{t-s}
\eeqn
with the convention that $\inf\{\emptyset\},\sup\{\emptyset\}=\infty$. As already noted in Neveu and Pitman \cite{NP89} (although under a slightly different form), 
the sequences $\{\tau_n\}_{n\geq1}$ and $\{\sigma_{n}\}_{n\geq1}$ play a key role in the tree $\Tr(\cT_f)$, being respectively related to the exploration times for the leaves and the branching
points respectively.

First,
the reader can easily convince herself that the set of finite $\{\tau_n\}$ coincide
with 
the completion times of all the sub-excursions nested in the function $f$ which are exactly of height $h$
(see  \cite{NP89} for more details). 
As already discussed, 
this implies that $\{p_f(\tau_n)\}_{n\geq1}$
contains  the set of leaves of the tree $\Tr(\cT_f)$.

Secondly, the very definition of $\sigma_i$'s implies that for every $m<n$,  $\inf_{[\tau_m,\tau_{n}]} f = f(\sigma_{k})$
for some $k\in\{m,\cdots,n\}$. From
(\ref{ionf}), this implies that
$p_f(\tau_n) \wedge p_f(\tau_{m})$ -- the most recent common ancestor of $p_f(\tau_n)$ and $p_f(\tau_{m})$ -- is given by
some  $p_f(\sigma_{k})$.

\bigskip

We now show that the times $\sigma_n(f)$ and $\tau_n(f)$ also
appear quite naturally in the two-sided Skorohod reflection. Recall from the introduction that $t_n$ (resp., $s_n$)
refer to the $n^{th}$ returning time (exit time) of $\La_{0,h}(f)$ at level $0$ (see (\ref{def:T})).

\begin{proposition}
\label{rtt} For every continuous function $f$ with $f(0)=0$
and 
for every $n\geq 1$,
\begin{enumerate}
\item
$\tau_n(f)$ is the $n^{th}$ returning time to level 0 of $\La_{0,h}(f)$, i.e., $\tau_n(f)=t_n(f)$.
\item
$\sigma_n(f)$ is the  $n^{th}$ exit time at level 0 of  $\La_{0,h}(f)$, i.e.,  $\sigma_n(f)=s_n(f)$.                                                                        
\item The function $f_h=f-\La_{0,h}(f)$ is non-decreasing (resp., non-increasing) on $[\sigma_{n+1}(f),\tau_{n+1}(f)]$ (resp., on $[\tau_n(f),\sigma_{n+1}(f)]$).
In particular, $\{f_h(\sigma_i)\}$ (resp., $\{f_h(\tau_i)\}$) coincide with the local minima (resp., local maxima) of 
$f_h$.
\end{enumerate}                                                     
\end{proposition}

As we shall see, this proposition is a consequence of elementary results on the two-sided 
Skorohod reflection that we now expose.
We start by introducing some notations: 
$$
\forall f, \ \forall t>0, \ R^T(f)(t) \ := \ 1_{t\geq T}  \left(f(t) - f(T)\right).
$$
In other words, $R^T(f)$ is constant on $[0,T]$
and follows the variation of $f$ afterwards.
The next elementary lemma states that 
the reflection of a path
can be obtained by successively reflecting the path up to some $T$
and then reflecting the remaining portion of the path from $T$ to $\infty$.

\begin{lemma}\label{two-steps-ref}
For any continuous function $f$ with $f(0)\in[0,h]$ and $T\geq0$,
\beqnn
\forall  t < T , \ \ \La_{0,h}(f)(t) & = &  \La_{0,h}\left(f(\cdot\wedge T)\right)(t), \\
\forall t\geq T, \ \ \La_{0,h}(f)(t) & = &   \La_{0,h}\left(R^T(f) +  \La_{0,h}(f)(T) \right)(t). \label{wws}
\eeqnn
\end{lemma}

\begin{proof}
In the following, we write
\beqnn
L^T(f)(t)	 & := & R^T(f)(t) +  \La_{0,h}(f)(T),
\eeqnn 
and for any continuous function $F$ with $F(0)\in[0,h]$, we denote by  $(c^0(F),c^h(F))$ the pair
of compensators solving the Skorohod equation for the two-sided reflection of $F$
on the interval $[0,h]$.
For $y=0,h$, define 
$$
\forall t\geq0,  \ \ \ \tilde c^y(t) \ : = \  \   c^y(f({\cdot\wedge T}))({t}) \ + \ c^y(L^T(f))(t).  
$$
We will show that  
$(\tilde c^0,\tilde c^h)$ solves the two-sided Skorohod equation  for $f$. 
We first need to prove that the function $G(t):=f(t) \ + \ \tilde c^h(t) + \tilde c^0(t)$
is valued in $[0,h]$. First
\begin{eqnarray*}
G(t)   & =  & \left(f({t\wedge T}) +   \sum_{y=0,h} c^y(f({\cdot \wedge T}))(t)\right)
      	\ + \ \left( L^T(f)(t) +  \sum_{y=0,h}  c^y(L^T(f))(t)  \right) - \La_{0,h}(f)(T)\\
     & = & \La_{0,h}(f({\cdot\wedge T}))(t) \ + \ \La_{0,h}\circ L^T(f)(t) \ - \La_{0,h}(f)(T) \\
  & = & \ \La_{0,h}(f)({t\wedge T}) \ + \ \La_{0,h}\circ L^T(f)(t) - \La_{0,h}(f)(T)
\end{eqnarray*}
where the first equality follows from the fact $f(t) = f({t\wedge T}) + L^T(f)(t) - \La_{0,h}(f)(T)$
and the last equality only states that the reflection of the function
$f({\cdot\wedge T})$ (the function $f$ ``stopped'' at $T$) 
is the reflection of $f$ stopped at $T$ (this can directly be checked 
from the definition
of the two-sided Skorohod reflection). 

The function $L^T(f)$ is constant and equal to $\La_{0,h}(f)(T)$
on the interval $[0,T]$. This easily implies that its reflection 
is also identically $\La_{0,h}(f)(T)$ on the same interval.
Thus, the latter equality implies that 
\beqn\label{htt}
 G(t)  := \left\{ \begin{array}{ll}
         \La_{0,h}(f)({t})                           & \mbox{if $t < T $},\\
        \La_{0,h}\circ L^T(f)(t)& \mbox{otherwise,}\end{array} \right. \label{barxi}
\eeqn

(\ref{htt}) implies that $G(t)$ 
belongs to $[0,h]$, hence proving that the first 
requirement of the Skorohod equation (see Theorem \ref{Skor}) is satisfied. 
The second requirement -- the function $\tilde c^{h}$ (resp., $\tilde c^0$) non-increasing (resp., non decreasing) --
is obviously satisfied since the function $\tilde c^h$ (resp., $\tilde c^0$) 
is constructed out of a compensator at $h$ (resp., at $0$).
Finally, for $y=0,h$, we need to show that the support of the measure
$d \tilde c^y$ is included in the set $G^{-1}(\{y\})$. In order to see 
that,
we use the fact that 
the support of the compensators $dc^{y}(L^T(f))$
and $d c^y(f({\cdot \wedge T}))$
are respectively included in $[T,\infty]$
and $[0,T]$ --- using the fact that if a function $g$ is constant 
on some interval, its compensator does not vary on this interval. 
As a consequence, for $y=0,h$
$$
\forall t\in [0,T], \ \   d \tilde c^y(t) \neq 0 \Longleftrightarrow d c^y(f({\cdot\wedge T}))(t) \neq 0.
$$
Further, $d c^y(f({\cdot\wedge T}))(t) \neq 0$  for $t$ such that $\La_{0,h}(f({\cdot\wedge T}))(t) = y$.
Since $\La_{0,h}(f({\cdot\wedge T}))$ and $G$ coincides on $[0,T]$ (by (\ref{barxi})),
we get that on $[0,T]$ the compensator $\tilde c^y_t$ only varies on  $G^{-1}(\{y\})$.
By an analogous argument, one can show that the same holds on the interval $[T,\infty]$.
Hence, 
the third and final requirement of the Skorohod equation holds for $\tilde c^y$, $y=0,h$.
This shows that $(\tilde c^0,\tilde c^h)$
solves the two-sided Skorohod reflection. Combining this with (\ref{barxi})
ends the proof of our lemma.

\end{proof}

Let $h\geq0$. 
For any continuous function $f$ with $f(0)\leq h$, let us define the one-sided reflection (with downward push) at $h$ -- denoted by $\Gamma^h(f)$ -- as
$$
\Ga^h(f) := f-(\sup_{[0,t]} f - h )\vee0.
$$
Along the same lines as the one-sided reflection at $0$
(as introduced in (\ref{one-sided})), the function $c(t)=-(\sup_{[0,t]} f - h )\vee0$
can be interpreted as
the minimal amount of downward push necessary to keep 
the path $f$ below level $h$. More precisely,
this function is easily seen to be the only continuous 
function $c$ with $c(0)=0$ satisfying the following requirements: (1) $f+c\leq h$, (2) c is non-increasing and, (3) $c$
does not vary off the set $\{t \ : \ f(t)+c(t) =h \}$. 

\begin{lemma}\label{pr-s}
For every $T\geq0$ and every continuous function $F$ 
with $F(0)\in[0,h]$ 
such that $\Ga^h(F)\geq 0$ (resp.,  $\Ga^0(F)\leq h$) on $[0,T]$, we must have
$$
\forall t\in[0,T], \ \  \Lambda_{0,h}(F)(t) \ = \ \Ga^h(F)(t) \ \ (\mbox{resp., }  \Lambda_{0,h}(F)(t) \ = \ \Ga^0(F)(t)).
$$
\end{lemma}

\begin{proof}
Let us consider a continuous $F$ with $F(0)\in[0,h]$  and 
such that $\Ga^h(F)\geq 0$ on $[0,T]$.
Let us prove that $\La_{0,h}(F)=\Ga^{h}(F)$ on $[0,T]$. 
We aim at showing that $(0,-(\sup_{[0,t]}F-h)^+)$
coincides with the pair of compensators of $F$ on the time interval $[0,T]$. 
First, 
$$
\Ga^h(F) \ = \  f \ + \  0 \ + (-(\sup_{[0,t]}F-h)^+)
$$
belongs to $[0,h]$ since $\Ga^h(F)\leq h$ and under the conditions of our lemma $\Ga^h(F)\geq0$. Secondly, 
using the fact that $-(\sup_{[0,t]}F-h)^+$
is the compensator for the one-sided case (at $h$), this function is non-increasing and only decreases 
when $\Ga^h(F)$ is at level $h$. This shows that  $\Ga^h(F)$ coincides with 
the two sided reflection of $f$ 
on the interval $[0,h]$. The case  $\Ga^0(F)\leq h$ can be handled similarly.
\end{proof}

\begin{proof}[Proof of Proposition \ref{rtt}] In order to prove Proposition \ref{rtt}, we will now proceed by induction  on $n$. 

{\bf Step 1.} 
We first claim that $\sigma_1\leq \theta_1$. When $\theta_1=\infty$, this is obvious. Let us assume that $\theta_1<\infty$. 
In order to see that, let us assume that $\sigma_1>\theta_1$. The definition
of $\theta_1$ implies that 
$\Ga^0(f)(\theta_1)=h$ and thus $\theta_1$
belongs to an excursion of $\Ga^0(f)$ away from $0$ (of height at least $h$),
whose interval we denote by $[t_-,t_+]$. 
Since $\sigma_1$ was defined as the last visit at $0$
of $\Ga^0(f)$ before time $\tau_1$ (see (\ref{t-s})) and $\sigma_1$
is assumed to be greater than $\theta_1$, 
$\sigma_1\geq t_+$ and the excursion of $\Ga^0(f)$ on $[t_-,t_+]$
must be completed before $\tau_1$. 
On the other hand,
\beqnn
h & = & \left(f(\theta_1) - \inf_{[0,\theta_1]} f\right) \ - \ \left(f(t_+)-\inf_{[0,t_+]} f\right) \\ 
   & = & f(\theta_1) - f(t_+) \\
   & \leq & \sup_{[\theta_1,t+]} f - f(t_+), 
\eeqnn
where we used the fact that $\inf_{[0,t]} f$ must be constant during an excursion of $\Ga^0(f)$ away
from $0$ in the second equality.
By continuity of $f$, there must exist $s\in[\theta_1 ,t_+]$
such that $\sup_{[\theta_1,s]}f -f(s)=h$, which implies that $\tau_1 \leq t_+$, thus yielding a contradiction
and proving that 
$\sigma_1\leq \theta_1$.

\medskip

Next,
the strategy for proving our proposition consists in breaking
the intervals $[0,\tau_1]$ into three pieces: $[0,\sigma_1]$, $[\sigma_1,\theta_1]$ and $[\theta_1,\tau_1]$. 
First, on $[0,\sigma_1]$, we must have $\Ga^{0}(f) < h$ since $\sigma_1< \theta_1$, and $\theta_1$
was defined as the first time $\Ga^0(f)(t)=h$.
By Lemma \ref{pr-s},
this implies that 
\beqn
\forall t\in [0,\sigma_1],\  \La_{0,h}(f)(t)=\Ga^0(f)(t) = f-\inf_{[0,t]} f, \ \ \ \mbox{and} \ \La_{0,h}(f)(\sigma_1)=0, \label{6-1}
\eeqn 
where the latter equality follows directly
from the definition of $\sigma_1$.
Next by Lemma \ref{two-steps-ref}, we must have
\beqnn
\forall t\in[\sigma_1,\infty], \  \La_{0,h}(f)(t) & = & \La_{0,h}(1_{\cdot\geq\sigma_1}(f - f(\sigma_1))(t).
\eeqnn
Using the fact that  
$\inf_{[0,t]} f$
remains constant during an excursion of $\Ga^0(f)$ away from $0$
and the fact that
$f-\inf_{[0,\cdot]}f<h$ on $[0,\sigma_1]$, 
it is easy to see that
$\theta_1$
coincides with the first visit of $1_{\cdot\geq\sigma_1}(f(\cdot)-f(\sigma_1))$ at $h$. 
Furthermore, since $\sigma_1$ is the {\it last} visit at $0$
of $\Ga^0(f)$ before $\tau_1$, we must have 
$$
\forall t\in (\sigma_1,\tau_1), \ \ \ f(t)-\inf_{[0,t]}f = f(t)-f(\sigma_1)>0.
$$
In particular, $f-f(\sigma_1)\in(0,h]$ on  the interval
$(\sigma_1,\theta_1]$ and thus,  $(0,0)$ solves the Skorohod equation for $1_{\cdot\geq\sigma_1}(f-f(\sigma_1))$ on this interval. This yields
\beqn
\forall t\in(\sigma_1,\theta_1), \ \La_{0,h}(f)(t)  \ = f(t) -  f(\sigma_1) >0.
\eeqn
Finally, using $\La_{0,h}(f)(\theta_1)=h$,  Lemma \ref{two-steps-ref} implies that 
\beqnn
\forall t\in[\theta_1,\tau_1], \ \  \Lambda_{0,h}(f)(t) 
									& = &  \Lambda_{0,h}(h +  1_{\cdot\geq \theta_1} (f(\cdot) - f(\theta_1)) \label{bordel}.
\eeqnn
A straightforward computation yields
\beqnn
\forall t\in[\theta_1,\tau_1], \ \Ga^h\left(h+1_{\cdot\geq \theta_1} (f(\cdot) - f(\theta_1)\right)(t)  & =    & f(t) + h - \sup_{[\theta_1,t]} f 
\eeqnn
By definition
of $\tau_1$, the RHS of the equality must remain positive on $[\theta_1,\tau_1)$. Using Lemma \ref{pr-s},
we get that 
\beqn
\forall t\in[\theta_1,\tau_1), \ \  \Lambda_{0,h}(f)(t)  \ = \ \Ga^h(f)(t) \ = f(t) +h - \sup_{[\theta_1,t]} f >0, \nonumber \\
\mbox{and} \ \ \Lambda_{0,h}(f)(\tau_1)  \ = 0
\label{6-3}
\eeqn
where the second equality follows from the  very definition of $\tau_1$.
Finally, combining (\ref{6-1})--(\ref{6-3}) yields
$$
\La_{0,h}(f)(t) = f \ - \ \left(1_{t\in[0,\sigma_1]} \cdot \inf_{[0,t]} f  \  + \ 1_{t\in[\sigma_1,\theta_1]} \cdot f(\sigma_1) \ + \ 1_{t\in[\theta_1,\tau_1]}\cdot (\sup_{[\theta_1,t]}f -h)   \right).
$$
As a consequence, $f_h=f-\La_{0,h}(f)$ is non-increasing (resp., non-decreasing) on $[0,\sigma_1]$ (resp., $[\sigma_1,\tau_1]$). Furthermore,
the argument above also shows that $\tau_1$ (resp., $\sigma_1$) is the first returning time (resp., exit time) at level $0$. Indeed, piecing together the previous results, we proved 
\beqnn
\La_{0,h}(f)(t)<h, \  \mbox{on $[0,\sigma_1)$ and } \La_{0,h}(f)(\sigma_1)=0 \\
\La_{0,h}(f)>0 \ \ \mbox{on $(\sigma_1,\tau_1)$}, \ \ \mbox{and}  \ \ \La_{0,h}(f)(\theta_1)=h, \ \La_{0,h}(f)(\tau_1)=0.
\eeqnn



\bigskip

{\bf Step n+1.} Let us assume Proposition \ref{rtt} is valid up to rank $n$. Recall that $R^{\tau_n}(f)=1_{t\geq\tau_n}(f(t)-f(\tau_n))$. 
By Lemma \ref{two-steps-ref},
$$
\forall t\in[\tau_n, \infty), \ \La_{0,h}(f) = \La_{0,h}(R^{\tau_n}(f)),
$$
where we used the induction
hypothesis to write $\La_{0,h}(f)(\tau_n)=0$. On the other hand, it is straightforward to check from the 
definitions of $\tau_{n+1},\theta_{n+1}$ and $\sigma_{n+1}$ in (\ref{t-s}) that 
$$
\sigma_{n+1}(f) \ = \ \sigma_1(R^{\tau_n}(f)),\ \ \tau_{n+1}(f) \ = \  \tau_{1}(R^{\tau_n}(f)),  \ \ \theta_{n+1}(f) = \theta_{1}(R^{\tau_n}(f)) 
$$
Applying the case $n=1$ to the function $R^{\tau_{n}}(f)$ immediately implies that our proposition is valid at step $n+1$.

\end{proof}

In order to prove Theorem \ref{main1},
we will combine Proposition {\ref{rtt} with the following lemma.

\begin{lemma}\label{isot}
Let $(\cT_1,d_1)$ and $(\cT_2,d_2)$ be two rooted real trees
with only finitely many leaves. For $k=1,2$, let $S_k=(z_1^k,\cdots,z_N^k)\in \cT_k$ 
such that the two following conditions hold.
\begin{enumerate}
\item For $k=1,2$, $S_k$ contains the leaves of $\cT_k$.
\item
$$
\forall i\leq N, \ d_1(\rho^1,z_i^1)=d_2(\rho^2,z_i^2) \ \mbox{and} \ \forall i,j \leq N, \ d_1(\rho^1, z_i^1\wedge z_j^1)=d_2(\rho^2, z_i^2 \wedge z_j^2),
$$ where  $\rho^k$ is the root of $\cT_k$.
\end{enumerate}
Under those conditions, there exists a root preserving isometry
from $\cT_1$ onto $\cT_2$.
\end{lemma}

\begin{proof}

For  $k=1,2$ and $m\leq N$, let $I_m^k:=[\rho^k,z_m^k]$.
Using the second assumption of our lemma, the two ancestral lines 
$I_m^1$ and $I_m^2$ must have the same length. 
From there, it easy to construct an isometry from $I_m^1$ onto $I_m^2$ as follows. Since $\cT_k$ is a tree,  
there exists a unique isometric map $\psi^{k}_m$ from
$[0,d_k(\rho^k,z_m^k)]$ onto $I_m^k$,
 such that $\psi^{k}_m(0)=\rho^k$ and $\psi^{k}_m(d_k(\rho^k,z_m^k))=z^k_m$. 
Define 
$$
\forall a\in I^1_m,  \ \  \phi_m(a) \ := \ \psi^{2}_m\circ(   \psi^{1}_m )^{-1}(a). 
$$ 
Since $d_1(\rho^1,z_m^1)= d_2(\rho^2,z_m^2)$, it is straightforward to show that 
$\phi_m$ defines an isometric isomorphism from $I^1_m$ to $I^2_m$.
Furthermore, 
$\phi_m$ preserves the root and $\phi_m(z^1_m)=z^2_m$.

Next, we claim that if $a\in I_m^1\cap I_l^1$, then $\phi_m(a)=\phi_l(a)$. 
We first show the property for $a=z_m^1\wedge z_l^1$.
Using the isometry of $\phi_m$ and the root preserving property,
we have
\beqnn
d_2(\rho^2, \phi_m(z_l^1\wedge z_m^1)) & = & d_1(\rho^1, z_l^1 \wedge z_m^1) \\
								& = & d_2(\rho^2, z_l^2 \wedge z_m^2), 
\eeqnn
where the second equality follows from the second assumption of our lemma.
Since $\phi_m(a)\in I_m^2$, it follows that $\phi_m(z_l^1\wedge z_m^1) = z_l^2\wedge z_m^2$ -- on a segment $I_m^k$, a point is uniquely determined by its distance to the root. By the same reasoning, we get that 
$\phi_l(z_l^1\wedge z_m^1) = z_l^2\wedge z_m^2$. Let us now take any point $a\in I_m^1\cap I_l^1$.
Under this assumption, we must have $a\preceq z_m^1\wedge z_l^1$.
Using the isometry property, this implies
$$
d_2(\rho^2,\phi_m(a)) \ =\  d_2(\rho^2,\phi_l(a))
$$ 
and 
$$
\phi_l(a), \phi_m(a) \preceq   z^2_m\wedge z^2_l.
$$
since we showed that $\phi_m( z^1_m\wedge z^1_l), \phi_l( z^1_m\wedge z^1_l) =  z^2_m\wedge z^2_l$. 
It easily follows that $\phi_m(a)=\phi_l(a)$, as claimed earlier.

We are now ready to construct the isometry from $\cT_1$ onto $\cT_2$.
First, for $k=1,2$, any point $a_k\in \cT_k$ must belong to some ancestral line of the form $[\rho^k,l]$,
for some leaf $l$ in the tree $\cT_k$. 
By what we just proved, and since $S_1$ contains all the leaves 
of $\cT_1$,
we can define the map $\phi$ from $\cT_1$ into $\cT_2$ as follows 
$$
\forall a\in \cT_1,  \ \ \phi(a) \ := \ \phi_m(a) \ \ \mbox{if $a\in I_m^1$}. 
$$
Since $S_2$ contains all the leaves of $\cT_2$, and 
any ancestral line of the form $[\rho^1,z^1_m]$
is mapped onto $[\rho^2,z_m^2]$, the map $\phi$ is onto. 

It remains to show that $\phi$ is isometric.
Let $a,b\in \cT_1$ and let us distinguish between two cases. First,  let us assume that $a$ and $b$ belong to the same 
ancestral line  $I_m^1$ for some $m\leq N$. Under this assumption, the property simply follows from the isometry of  $\phi_m$. 
Let us now consider the case
where $a$ and $b$ belong to two distinct ancestral lines:
$a\in I_m^1$ but $a\notin I_l^1 $, 
and $b\in I_l^1$ but $b\notin I_m^1$; in such a way that $a\wedge b=z_l^1\wedge z_m^1$. Using the fact that both
$\phi_m$ and $\phi_l$ are isometric,
and  $\phi(z_l^1\wedge z_m^1)=z_l^2\wedge z_m^2$, and $\phi_l(a)\in I_l^2,\phi_m(a)\in I_m^2$  we get that 
$\phi(a) \wedge \phi(b)=z_l^2\wedge z_m^2 $.
We can then write $[\phi(a),\phi(b)]$ as the union $[z_l^2\wedge z_m^2, \phi(a)]\cup[z_l^2\wedge z_m^2, \phi(b)]$ and 
write 
\beqnn
d_2(\phi(a),\phi(b)) & = & d_2(z_l^2\wedge z_m^2,\phi(a))+d_2(z_l^2\wedge z_m^2,\phi(b)) \\
			       & = & d_1(z_l^1\wedge z_m^1,a)+d_1(z_l^1\wedge z_m^1,b) \\
			       & = & d_1(a,b) 
\eeqnn
where the second equality follows by applying the previous case to the pairs of
points $(z_l^1\wedge z_m^1,a)$ and
$(z_l^1\wedge z_m^1,b)$.


\end{proof}

In the following,
we make the assumption that the $h$-trimming
of the tree $\cT_f$
is non-empty,
i.e., that $\sup_{[0,\infty)} f <h$.

\begin{proof}[Proof of Theorem \ref{main1}]

Recall that $C_0^+(\R^+)$
denotes the set of continuous non-negative functions  
with $f(0)=0$
and compact support.
We start by showing the first item of our theorem, i.e., that 
$f\in C_0^+(\R^+)$
implies that $f_h\in C_0^+(\R^+)$.
First, as an easy corollary of Proposition \ref{rtt}, we get that for every $f\in C_0^+(\R^+)$,
the function $f_h=f-\La_{0,h}(f)$ is non-negative. This simply follows from the fact
that  
the local minima of $f_h$ are attained on the set $\{\sigma_i\}$, on which
$f(\sigma_i)=f_h(\sigma_i)$ since $\La_{0,h}(f)(\sigma_i)=0$. Since $f(\sigma_i)\geq0$, the function
$f_h$ is non-negative. Secondly, the function $f_h$ must have compact support. In order to see 
that, let us take $K$ such that $\forall t\geq K, f(t)=0$. For $t\geq K$, 
$f_h(t)=-\La_{0,h}(f)$ and since $f_h\geq0$ and $-\La^{0,h}(f)\leq0$, it 
follows that  $f_h\equiv0$ after time $K$.

Next,  let us show that
$f_h$
is the contour function of the $h$-trimming of the tree $\cT_f$ (up to
an isomtetric isomorphism preserving the root).
For $k<N_h(f)$, Proposition \ref{rtt}
immediately  implies that the maximum of $f_h$ on $[\sigma_{k},\sigma_{k+1})$
is attained at time $\tau_k$ and that
the set
$$
I_k:=\{t \in [\sigma_k,\sigma_{k+1}] \ : \ f(t)=f(\tau_k)  \ \}
$$ 
is a closed interval. On the one hand, any time $t\in[\sigma_k,\sigma_{k+1}]$
outside of this interval 
is the starting or ending time of a sub-excursion with (strictly) positive height, and for such $t$, $p_{f_h}(t)$ can not be a leaf. On the other hand, 
we have $p_f(t)=p_f(t')$ for $t,t'\in I_k$. This implies that the only possible leaf visited during the time interval $[\sigma_{k},\sigma_{k+1}]$
is given by $p_{f_h}(\tau_k)$ and thus, that
the set of leaves of $\cT_{f_h}$ is included in the finite set of points $\{p_{f_h}(\tau_n)\}_{n\leq N_h(f)}$.
 
 
\medskip

As explained at the beginning of this section,  any leaf of the tree $\Tr(\cT_f)$
must be explored at some $\tau_n$, i.e., the set of leaves
of $\Tr(\cT_f)$
is a subset of $\{p_f(\tau_n)\}_{n\leq N_h(f)}$.
In order to prove our result, we use Lemma \ref{isot}
with $z^1_i=p_f(\tau_i)$ and $z_i^2=p_{f_h}(\tau_i)$
and $N=N_h(f)$.
First,
item 1. of Proposition \ref{rtt} implies that the height 
of the vertices $p_f(\tau_i)$ and $p_{f_h}(\tau_i)$ are identical, i.e., that $f(\tau_i)=f_h(\tau_i)$
since $\La_{0,h}(f)(\tau_i)=0$.
In order to show that $\Tr(\cT_f)$ 
and $\cT_{f_h}$ are identical (up to a root preserving isomorphism), 
it is sufficient to check that 
$$
\forall i < j, \ \ \inf_{[\tau_i,\tau_j]} f = \inf_{[\tau_i,\tau_j]} f_h,
$$
i.e., that the height of the most recent common ancestor of the vertices visited at $\tau_i$ and $\tau_j$ -- respectively
the $i^{th}$ and $j^{th}$ leaf --
is the same in both trees.
To justify the latter relation, we first note that 
the definition of the $\sigma_m$'s (see (\ref{t-s})) implies that
$\inf_{[\tau_i,\tau_j]} f$ must be attained at some $\sigma_k$ (for some
$k\in\{i+1,\cdots,j\}$). On the other hand, the third item of the  
Proposition \ref{rtt} implies that 
the same must hold for $f_h$ since the set of local minima of $f_h$ coincide with $\{f_h(\sigma_i)\}$. 
Since 
$
f(\sigma_i)=f_h(\sigma_i),
$ 
($s_i=\sigma_i$ by the second item of Proposition \ref{rtt} and $\La^{0,h}(f)(s_i)=0$), Theorem \ref{main1} follows. 

\end{proof}

\bigskip

\begin{proof}[Proof of Proposition \ref{algo}] 
Let us define 
$$
\cT_{f_h}^n \ := \ \{z\in\cT_{f_h}: \ \exists t\leq \tau_n, \ p_{f_h}(t)=z\},
$$
the set of vertices in $\cT_{f_h}$ visited up to time $\tau_n$.
$\cT_{f_h}$ can be constructed recursively by adding to $\cT^n_{f_h}$
all the vertices in $\cT^{n+1}_{f_h}\setminus \cT^{n}_{f_h}$ for every $n<N$,
where $N\equiv N_h(f)$. 
(Indeed, by definition of a real tree from its contour path,
if a point is visited at a given time, its ancestral line must have been explored before 
that time. Thus, if all the leaves have been explored at a given time -- e.g., at $\tau_N$-- every vertex has been visited 
at least once before that.)
For every $n<N$, let us  show that the  set $\cT^{n+1}_{f_h}\setminus \cT^{n}_{f_h}$ is a branch (i.e., a segment $[a,b]$
with $a,b\in{\cT_{f_h}}$ and $a\preceq b$) such that
\begin{enumerate} 
\ii[(i)] with tip $b=p_{f_h}(\tau_{n+1})$ -- i.e., 
$\forall z\in  \cT_{f_h}^{n+1}\setminus \cT^{n}_{f_h}, \ z\preceq p_{f_h}(\tau_{n+1})$ --
\ii[(ii)] attached to $a=p_{f_h}(\sigma_{n+1})$ -- i.e., $\forall z\in  \cT_{f_h}^{n+1}\setminus \cT^{n}_{f_h}, \ p_{f_h}(\sigma_{n+1}) \preceq z$.
\ii[(iii)] $p_f(\sigma_{n+1})$ belongs
to the ancestral line
$[\rho_{f_h},p_{f_h}(\tau_n)]$.
\end{enumerate}
Let $t\in[\tau_n,\tau_{n+1}]$ and let $n<N$.
The definition of the real tree ${\cal T}_{f_h}$ implies  
that 
the point $p_{f_h}(t)$
has been visited before $\tau_n$ if and only if there exists $s\leq\tau_n$ such that 
\be\label{gt-gs}
\inf_{[s,t]} f_h = f_h(t) = f_h(s),
\ee
i.e., $t$ must be the ending time of a sub-excursion straddling $\tau_{n}$.
On the one hand, 
since $f_h$ is non-increasing on 
$[\tau_n,\sigma_{n+1}]$, we have
\be\label{anc}
\forall t\in[\tau_n,\sigma_{n+1}],\  \inf_{[\tau_{n},t]} f_h = f_h(t). 
\ee
Since $f_h(0)=0$ and $f_h(t)\geq 0$
one can find $s\leq \tau_n$ such that (\ref{gt-gs}) is satisfied (using the continuity of $f_h$).
Thus, every point visited on the time interval $[\tau_n,\sigma_{n+1}]$ has already been visited before $\tau_{n}$ and thus does not belong to $\cT_{f_h}^{n+1}\setminus \cT_{f_h}^{n}$.

On the other hand,
the function $f_h$ is non-decreasing on $[\sigma_{n+1},\tau_{n+1}]$. Let us define 
$$
\bar \theta_{n+1}=\sup\{t\in[\sigma_{n+1},\tau_{n+1}] \ : \ f_h(t)=f_h(\sigma_{n+1}) \}
$$
(with the convention $\sup\{\emptyset\}=\tau_{n+1}$). First, 
the definition of our real tree ${\cal T}_{f_h}$ implies that for
 any $p_{f_h}(t)$ 
with  $t\in[\sigma_{n+1}, \bar \theta_{n+1}]$
coincides with $p_{f_h}(\sigma_{n+1})$. Secondly, for any $t\in [\bar \theta_{n+1},\tau_{n+1}]$, and any $s\leq\tau_{n}$
$$
\inf_{[s,t]} f_h \leq f_h(\sigma_{n+1})<f_h(t) 
$$
which implies that any point visited during the interval $(\bar \theta_{n+1},\tau_{n+1}]$ belongs to ${\cal T}_{f_h}^{n+1}\setminus {\cal T}_{f_h}^n$. Furthermore,
the previous inequality implies that
\be\label{deux}
\forall t\in(\bar \theta_{n+1},\tau_{n+1}], \ \ p_{f_h}(t)\succ p_{f_h}(\sigma_{n+1}).
\ee 
Finally,
$$
\forall t\in[\bar \theta_{n+1},\tau_{n+1}],\  \inf_{[t,\tau_{n+1}]} f_h =  f_h(t),
$$
which implies that
\be\label{trois}
\forall t\in[\bar \theta_{n+1},\tau_{n+1}], \ \ p_{f_h}(t) \preceq p_{f_h}(\tau_{n+1}).
\ee
Combining the results above, we showed the claims (i)--(iii) made earlier:
$\cT_{f_h}^{n+1}\setminus \cT_{f_h}^{n}$ is a branch
of extremity $p_f(\tau_{n+1})$ (see (\ref{trois})) attached at 
the point $p_{f_h}(\sigma_{n+1})$ (see (\ref{deux})), which belongs to $[\rho_{f_h},p_{f_h}(\tau_n)]$ (see (\ref{anc}) 
applied to $t=\sigma_{n+1}$). 
Furthermore, 
the length of the branch is given by
\be\label{xn}
X_{n+1} := f_h(\tau_{n+1}) - f_h(\sigma_{n+1}),
\ee
(height of the $(n+1)^{th}$ leaf $-$ height of the attachment point) and the distance of the attachment point from the leaf $p_{f_h}(\tau_n)$
is given by
\be\label{yn}
Y_{n+1} := f_h(\tau_{n}) - f_h(\sigma_{n+1}).
\ee
(Height of the $n^{th}$ leaf $-$ height of the attachment point.)

\end{proof}

\section{Proof of Theorem \ref{PN}}


Next, let $e$ be a Brownian excursion conditioned on having a height greater than $h$
and let $\{(X_n(e),Y_n(e))\}_{i\leq N_h(e)}$ be defined as in (\ref{xy}), i.e.,
\beqnn
\forall n\geq1, & X_n(e) = &   e_h(t_{n}) - e_h(s_{n})  \nonumber \\
		      & Y_n(e) =                       &    e_h(t_{n-1})-e_h(s_{n}),
\eeqnn
with $t_n\equiv t_n(e)$, $s_n\equiv s_n(e)$ and let $N_h(e)$ be the number of returns of $e$ to level $0$. 
As discussed in the introduction (see the discussion preceding Theorem \ref{PN}), in order to prove that the trimmed tree $\Tr({\cal T}_e)$
is a binary tree, we need to show that
$\{(X_i(e),Y_i(e))\}_{i\leq N_h(e)}$ is identical in law with a sequence $\{(\tilde X_i,\tilde Y_i)\}_{i\leq \tilde N}$,
where $\{(\tilde X_i,\tilde Y_i)\}_{i\in\N}$
is an infinite sequence of independent exponential random variables 
with mean $h/2$ and 
$$
\tilde N \ := \ \inf\{n\geq0 \ :  \ \sum_{i=1}^{n} (\tilde X_{i}- \tilde Y_{i+1}) < 0\}.
$$

The idea of the proof  consists in constructing a coupling
between $\{(X_i(e),Y_i(e))\}_{i\leq N_h(e)}$
and  $(\{(\tilde X_i, \tilde Y_i)\}_{i\geq 1}, \tilde N)$ as follows.
Let $w$ be a  Brownian motion with $w(0)=0$,  independent of
the excursion $e$, and define
\be\label{K(e)}
\tilde w(t) \ := \ e(t) + w\left((t-K(e))\vee0\right) \ \ \mbox{where $K(e):=\sup\{t>0: \ e(t)>0\}$},
\ee 
obtained by pasting the process $w$ at the end of the excursion $e$.
Finally,
define  $\tilde X_n:=X_n(\tilde w)$ and $\tilde Y_n:=Y_n(\tilde w)$.

It is easy to show
that the support of $\La_{0,h}(e)$
is included in the support of $e$, that we denote by $[0,K(e)]$. (This was established in the course of proving Theorem \ref{main1}.)
As a consequence,
for every $n\leq N_h(e)$,
we must have
$t_n(e),s_n(e)\leq K(e)$ (recall that for $n\leq N_h(e)$, $t_n(e)$ and $s_{n}(e)$ coincide with the $n^{th}$ {\it finite} 
returning and exit times at $0$). Since 
$e$ and $\tilde w$ (and their reflections) 
coincide up to $K(e)$, this implies that
$s_n(e)=s_n(\tilde w), t_n(e)=t_n(\tilde w)$ and that
$X_n(e)=\tilde X_n,Y_n(e)=\tilde Y_n$ for $n\leq N_h(e)$.
 Theorem \ref{PN}
is a direct consequence of our coupling and the two following lemmas.

\begin{lemma}\label{l1}
\begin{enumerate} 
\item $\tilde X_1,\tilde Y_2, \tilde X_2, \tilde Y_3,\cdots$ is an i.i.d. sequence of pair of independent exponential variables with parameter $h/2$.
Further, $\tilde Y_1=0$.
\item Under our coupling, for $1\leq i\leq N_h(e)$, 
\begin{itemize}
\item $\tilde X_i \ = \ l^h(\tilde w)(\tilde t_i)-l^h(\tilde w)(\tilde t_{i-1})$,
\item $\tilde Y_i\ = \ l^0(\tilde w)(\tilde t_i)-l^0(\tilde w)(\tilde t_{i-1})$,
\end{itemize}
where $\tilde t_i := t_i(\tilde w)$.
\end{enumerate}
\end{lemma}

\begin{lemma}\label{l2}
Under our coupling,
$$
N_h(e) = \ \inf\{n : \sum_{i=1}^{n} (\tilde X_{i}- \tilde Y_{i+1}) < 0\}  \ \ \mbox{a.s.}.
$$
\end{lemma}


\begin{proof}[Proof of Lemma \ref{l1}]

Let us first prove that $\tilde Y_1=0$ and that $\tilde Y_1=l^0(\tilde t_{1})-l^0(\tilde t_0)$.
Let 
$$
\bar T_1:=\inf\{t: \ \tilde w(t)=e(t)=h\}.
$$
Since $e$ is a Brownian 
excursion with height larger than $h$,
$\bar T_1<\infty$ and 
$\tilde w\in(0,h]$ on $(0,\bar T_1]$. From there, it immediately
follows that $\La_{0,h}(\tilde w)=\tilde w$ on $[0,\bar T_1]$
and that $\bar T_1$ is the first returning time at level $h$
for the reflected process, i.e., $\bar T_1 = T_1(\tilde w)$ (see (\ref{def:T}) for a definition
of $T_1(\tilde w)$). 
Further, 
$\tilde s_1$ -- the first exit time of $\La_{0,h}(\tilde w)$ at level $0$ --
is equal to $0$.
Since $c^h(\tilde w)$ does not vary off the set $\{t \ : \ \La^{0,h}(\tilde w)(t)=h\}$, we have
\beqnn
\tilde Y_1 & = &  -(c^0+c^h)(\tilde w)(0)+(c^0+c^h)(\tilde w)(\tilde s_1) \\
	        & = &  -c^0(\tilde w)(0)+c^0(\tilde w)(\tilde s_1),
\eeqnn
implying that $\tilde Y_1=0$.
Finally, we also get that $l^0(\tilde t_{1})-l^0(\tilde t_0)=0$,
since
$\tilde t_1$ coincides with the first returning time of the 
reflected process at $0$, and this process
never hits $0$ on the interval $(0,\tilde t_1)$.

Before proceeding with 
the rest of the proof,
we start with a preliminary discussion.
Let $w'$ be a one-dimensional Brownian motion starting at $x\in[0,h]$.
Recall that the one-sided Skorohod reflection $\Ga^0(w')$ is distributed as the absolute value
of a standard Brownian motion,  
and the compensator
$c(w')(t):=(-\inf_{[0,t]} w')^+$ is the local time
at $0$ of $\Ga^0(w')$. A proof of this statement can be found in \cite{KS91}. By following the exact same steps,
one can prove an analogous statement for the two-sided case, i.e., that for any Brownian motion $w'$ starting at some $x\in[0,h]$,
$\La_{0,h}(w')$ is identical in law with $\lambda_{0,h}(w')$ (the standard reflection of the Brownian motion $w'$ -- see Section \ref{Intro} in the discussion preceding
Theorem \ref{teo1} for a description of $\lambda_{0,h}(w')$)
and that $-c^h(w')$ and $c^0(w')$ are respectively the local times at $h$ and $0$ of this process. 

Next, let us define $\tilde t_n=t_n(\tilde w)$
and $\tilde s_n=s_{n}(\tilde w)$
and recall that  the $h$-cut $\tilde w_h$ is defined as
$$
\tilde w_h = \tilde w - \La_{0,h}(\tilde w) = -c^0(\tilde w) - c^h(\tilde w).
$$
By definition of $\tilde X_n$, we have 
\beqn
\tilde X_{n}
& = &   \tilde w_h(\tilde t_{n}) -  \tilde w_h(\tilde s_{n})  \nonumber \\
& = &  c^h(\tilde w)(\tilde s_{n}) -  c^h(\tilde w)(\tilde t_{n})  \nonumber \\
& = & c^h(\tilde w)(\tilde t_{n-1}) -  c^h(\tilde w)(\tilde t_{n}).  \label{local-x}
\eeqn
The second line follows from the fact that 
$c^0(\tilde w)$ does not vary off the set $\La_{0,h}(\tilde w)^{-1}(\{0\})$ and $\La_{0,h}(\tilde w)>0$ on $(\tilde s_{n},\tilde t_n)$; 
the third line is a consequence of the fact that
 $c^h(\tilde w)$ does not vary off the set $\La_{0,h}(\tilde w)^{-1}(\{h\})$ and 
 $\La_{0,h}(\tilde w)<h$ on $(\tilde t_{n-1},\tilde s_n)$.
By an analogous argument, one can prove that
\beqn
\tilde Y_n & = & c^0(\tilde w)(\tilde t_n)  -c^0(\tilde w)(\tilde t_{n-1}). \label{local-y}
\eeqn

With those results at hand, we are now ready to prove our lemma. 
In the first paragraph, we already argued that
$\La_{0,h}(\tilde w)=\tilde w$ on $[0,\bar T_1]$. 
By Lemma \ref{two-steps-ref},
\beqnn
\forall t\geq0, \ w'(t) & :=& \Lambda_{0,h}(\tilde w)(t+\bar T_1) \\
        & = &  \Lambda_{0,h}(\tilde w(\cdot+\bar T_1))(t). 
\eeqnn
The strong Markov property and the discussion above imply 
that the path $w'$ 
is identical in law with a reflected 
Brownian motion  (where the reflection is a two-sided ``standard reflection")
starting at level $h$. Further, the compensators $c^0(\tilde w)(\bar T_1+\cdot)$ and $-c^h(\tilde w)(\bar T_1+\cdot)$
are the local times at $0$ and $h$ for the process $w'$. Using (\ref{local-x})--(\ref{local-y}), we easily obtain that 
$\tilde X_n$ (resp., $\tilde Y_{n}$) is the local time
accumulated at $h$ (resp., $0$), for $1\leq n \leq N_h(e)$ (resp., $2\leq n \leq N_h(e)$) on $[\tilde t_{n-1},\tilde t_n]$.
This completes the proof of the second part of our lemma. 
Finally,
by standard excursion theory, $\tilde X_1, \tilde Y_2, \tilde X_2, \tilde Y_3,\cdots$ are i.i.d. exponential random variables with mean $h/2$. Independence follows from the strong Markov property, whereas $\tilde X_i$ and $\tilde Y_{i+1}$ are distributed as the amount of Brownian local time accumulated at $0$ before occurrence of an excursion of height larger or equal to $h$.

\end{proof}


\begin{proof}[Proof of Lemma \ref{l2}]

Recall that 
$$
\tilde X_n = \tilde w_h(\tilde t_n)-\tilde w_h(\tilde s_n)  \ 
\mbox{and} \ \tilde Y_n = \tilde w_h(\tilde t_{n-1}) - \tilde w_h(\tilde s_n),
$$
where we wrote  $\tilde t_n=t_n(\tilde w)$, $\tilde s_k = s_k(\tilde w)$. Thus,
\beqnn
\tilde w_h(\tilde s_{n+1}) 
& = & 
[\sum_{i=1}^n (\tilde w_h(\tilde t_i)- \tilde w_h(\tilde t_{i-1}))]  \ - \  (\tilde w_h(\tilde t_{n}) - \tilde w_h(\tilde s_{n+1})) \\
& = &
[(-\tilde Y_1 + \tilde X_1) + (-\tilde Y_2+\tilde X_2) + \cdots +(-\tilde Y_{n}+\tilde X_{n})] \ - \tilde Y_{n+1} \\
& = &
-\tilde Y_1 + \sum_{i=1}^{n} (\tilde X_i-\tilde Y_{i+1})\\
& = & \sum_{i=1}^{n} (\tilde X_i-\tilde Y_{i+1}),
\eeqnn
where the last equality follows from the fact that  $\tilde Y_1=0$ (see the first item of the previous lemma).

Let us now show that $N_h(e)=\inf\{n : \sum_{i=1}^{n} (\tilde X_{i}- \tilde Y_{i+1}) < 0\}$ a.s..
First, let us take $n< N_h(e)$. 
On the one hand, we  already argued that $\tilde s_{n+1}\leq K(e)$. On the other hand,
$\tilde w_h\geq0$ on $[0,K(e)]$ since $e_h\geq0$ (by  Theorem \ref{main1}) and that $e_h$ and $\tilde w_h$
coincide up to $K(e)$. Thus, 
$$
\tilde w_h(\tilde s_{n+1}) = \sum_{i=1}^{n} (\tilde X_i-\tilde Y_{i+1}) \geq 0.
$$
Conversely,
let us take $n= N_h(e)$. By Proposition \ref{rtt},  
$\tilde w_h$ attains a minimum at $\tilde s_{\tilde N_h(e)+1}$ on the interval $[\tilde t_{N_h(e)}, \tilde t_{N_h(e)+1}]$. 
Using the fact that 
$$
\sup \ \mbox{Supp}(e) \geq \sup \ \mbox{Supp}(\Ga_{0,h}(e)) 
$$
(see the proof of Theorem \ref{main1}(1)),
it is easy to show that $K(e)\in[\tilde t_{N_h(e)}, \tilde t_{N_h(e)+1}]$ which implies that
\beqnn
 0 = \tilde w_h(K(e)) & \geq &  \tilde w_h(\tilde s_{N_{h}(e)+1}) \\
 		      &   =   &   \sum_{i=1}^{N_h(e)} (\tilde X_i - \tilde Y_{i+1}). 
 \eeqnn
Since $\tilde X_i$ and $\tilde Y_i$ are exponential 
random variables, this inequality is strict almost surely. This completes the proof of the lemma.
\end{proof}

\section{Proof of Theorem \ref{teo1}}
\label{proof-sect}

Let $(z^\theta, w)$
be a weak solution of (\ref{sticky-bm}). 
Our proof builds on the approach of Warren \cite{W02}. In this work, it is proved that the pair $(z^\theta,w)$ can be constructed by adding some extra noise to
the reflected process 
$$
\xi(t):= w(t)-\inf_{[0,t]} w
$$ 
as follows. First, there exists a unique $\sigma$-finite measure -- here denoted by ${\cal L}_\xi$ and  
referred to as the branch length measure --  on the metric space $(\TT_\xi,d_\xi)$ 
such that 
$$
\forall a,b\in\TT_\xi, \ \ {\cal L}_\xi([a,b]) \ = \ d_\xi([a,b]).
$$
(See \cite{E05} for more details).
Conditioned on a realization of $\xi$,
let us now consider 
the Poisson point process on $(\cT_\xi,d_\xi)$ with intensity measure
$2\theta {\cal L}_\xi$
and define the pruned tree
$$
\cT^\theta_\xi \ := \ \{z\in\cT_\xi \ : \ [\rho_\xi,z] \ \mbox{is unmarked}\},
$$ 
obtained after removal of every vertex with a marked ancestor along its ancestral line.
Finally, define $z^\theta(t)$ as the distance of the point $p_\xi(t)$ 
from
the subset $\cT^\theta_\xi$, i.e.,
\be
z^\theta(t) \ := \ \left\{ \begin{array}{c} 
\mbox{0 if $p_\xi(t)\in \cT^\theta_\xi$,}
\\
\mbox{$\xi(t) - A(t)$ otherwise,} \end{array} \right. \label{ztheta}
\ee
where $A(t)=0$ if there is no mark along the ancestral line $[\rho_\xi,p_\xi(t)]$, and is equal to the height of the first mark (counted from the root) on $[\rho_\xi,p_\xi(t)]$ otherwise. Informally,
$(z^\theta(t); t\geq0)$
can be thought of as the exploration process ÒaboveÓ
the pruned tree $\cT^\theta_\xi$ -- see Fig \ref{sticky-fig}.

\begin{figure}[ht]
   \centering
      \includegraphics[width=0.3\textwidth]{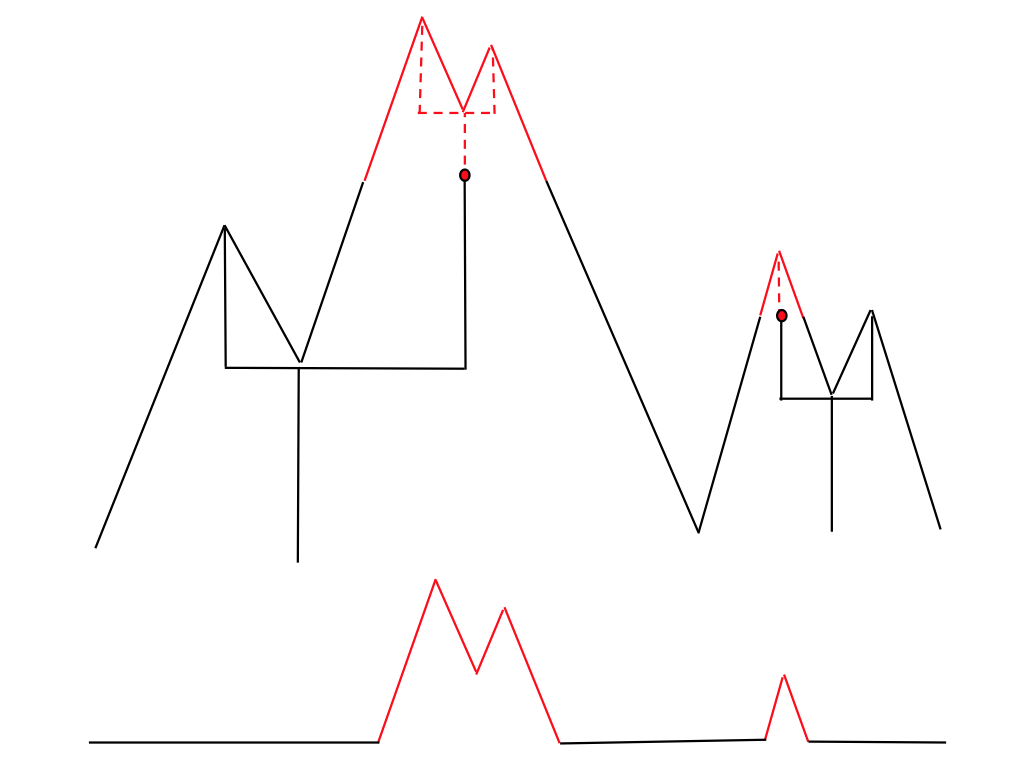}
  \caption{{\it The top panel displays a reflected random walk $\xi$ with marking of the underlying tree $\cT_\xi$.
 $\cT_\xi^\theta$ is the black subset of the tree.
The bottom panel
  displays the sticky path $(z^\theta(t); t\geq0)$ obtained by concatenating
  the contour paths
  of the red subtrees attached to $\cT_\xi^\theta$.}}\label{sticky-fig}
\end{figure}

\begin{theorem}[\cite{W02}]\label{wae}
The process
$(z^\theta(t),w(t); \ t\geq 0)$ is a weak solution of the SDE (\ref{sticky-bm}). 
\end{theorem}

Using this result, we proceed with the proof of Theorem \ref{teo1}.
Let $\cT^{(t)}_\xi:=\cT_\xi \cap \{p_\xi(s) : \ s\leq t\}$ be the sub-tree consisting of all the vertices in $\cT_\xi$
visited up to time $t$. 
For every $s$, the set
$$
\{x\in \cT_\xi \ : \ x\preceq p_\xi(s) \mbox{ and }d_\xi(x,p_{\xi(s)})\geq h\}.
$$
is totally ordered.
We define $a_h(s)$ as the $\sup$
of this set,
with the convention that $\sup\{\emptyset\}=\rho_\xi$. Informelly,
$a_h(s)$ is the ancestor of $p_\xi(s)$ at a distance $h$.
Following the construction of the pair $(z^\theta,\xi)$ described earlier,
$\sup_{[0,t]} z^\theta\leq h$ if and only if 
$$
\forall s\leq t, \ \ [\rho_\xi, a_h(s)]  \ \ \mbox{is unmarked},
$$
which is easily seen to be equivalent to not finding any mark
on the $h$-trimming of the 
tree $\cT^{(t)}_\xi$.
By a standard result about Poisson point processes, we have
\beqn
\P(\sup_{[0,t]} z^\theta \leq h \ | \ \sigma(w)) & = & \P\left( \Tr(\cT_\xi^{(t)}) \ \ \mbox{is unmarked} \ | \ \sigma(w)\right) \nonumber \\
							& = & \exp\left[-2\theta \cdot {\cal L}_\xi( \Tr(\cT_\xi^{(t)}) )\right] \label{eqp}.
\eeqn
Let us define
$$
\xi^{(t)}(s) \ := \ \left\{ \begin{array}{c} \xi(s) \ \mbox{if $s\leq t$} \\  (\xi(t) - (s-t))^+ \ \ \mbox{otherwise.} \end{array} \right.
$$
which is a function in $C_0^+(\R^+)$ from which we can construct the real rooted
tree $(\cT_{\xi^{(t)}},d_{\xi^{(t)}})$.

\begin{lemma}
There exists an isometric isomorphism preserving the root from 
$(\cT_\xi^{(t)},d_{\cT_\xi})$ onto $(\cT_{\xi^{(t)}}, d_{\cT_{\xi^{(t)}}})$.
\end{lemma}
\begin{proof}
To simplify the notation, we write ${\cal T}_1=\cT_\xi^{(t)}$ and ${\cal T}_2=\cT_{\xi^{(t)}}$. 
For every $y\in\cT_1$, define  $t_y$ to be the minimal element of the fiber $\{s \ : \ p_{\xi}(s)=y\}$
(i.e., the first exploration time for $y$)
and define 
the mapping
\beqnn
\phi \ : \ \cT_1 & \rightarrow & \cT_2 \\
	     y   & \rightarrow & p_{\xi^{(t)}}(t_y).  	
\eeqnn
It is straightforward to show that $\phi$ defines a mapping from ${\cT_1}$ to $\cT_2$ preserving the root.
We first show that $\phi$ is surjective. 
In order to do so, we start by showing that 
\be\label{chj}
\forall s\leq t, \  \phi(p_{\xi}(s)) = p_{\xi^{(t)}}(s).
\ee
For
$s\leq t$, we have
\beqnn
t_{p_\xi(s)} & = & \inf\{ u \ : \ p_\xi(u) = p_\xi(s)   \}  \\
	       & = & \inf\{u\leq s \ : \ \xi(u)=\xi(s)=\inf_{[u,s]} \xi \} \\
	        & = & \inf\{u\leq s \ : \ \xi^{(t)}(u)=\xi^{(t)}(s)=\inf_{[u,s]} \xi^{(t)} \}, 
\eeqnn
where the last equality follows from the fact that $\xi$ and $\xi^{(t)}$
coincide before time $t$. The latter identity implies
that
$
t_{p_\xi(s)} \in \{u\leq s \ : \ \xi^{(t)}(u)=\xi^{(t)}(s)=\inf_{[u,s]} \xi^{(t)} \} 
$ 
or equivalently that 
$$
p_{\xi^{(t)}}(t_{p_\xi(s)}) = p_{\xi^{(t)}}(s),
$$
which can be rewritten as
(\ref{chj}), 
as claimed earlier. In order to show surjectivity,
let us take $v\in\TT_2$ and
$s$ such that $v=p_{\xi^{(t)}}(s)$. We distinguish between two cases: 
(1) if $s\leq t$,
the previous result immediately implies that $v\in\phi(\cT_1)$, and
(2) $s> t$, 
since $\xi^{(t)}$
is continuous and non-increasing on $[t,\infty)$,
one can find
$s' \leq t$ such that 
$$
\xi^{(t)}(s)=\xi^{(t)}(s')=\inf_{[s',s]} \xi^{(t)}
$$
implying that $v=p_{\xi^{(t)}}(s')$ and we are back to case (1).

\medskip

It remains to show that $\phi$ is an isometry. 
Let $x_1,x_2\in\cT_1$. We have
$
\phi(x_i)=p_{\xi^{(t)}}(t_{x_i})
$
with $t_{x_i}\leq t$. Since $\xi$ and $\xi^{(t)}$
coincide up to time $t$,
we must have 
\beqnn
d_{\cT_2}(\phi(x_1),\phi(x_2))) & = & \xi^{(t)}(t_{x_2}) \ + \ \xi^{(t)}(t_{x_2}) \ - \ 2  \inf_{[t_{x_1}\wedge t_{x_2}, t_{x_1}\vee t_{x_2}]} \xi^{(t)} \\
					      & = & \xi(t_{x_2})\ +\ \xi(t_{x_1}) \ - \ 2  \inf_{[t_{x_1}\wedge t_{x_2}, t_{x_1}\vee t_{x_2}]} \xi \\
					      & =&  d_{\cT_1}(x_1,x_2).
\eeqnn

\end{proof}
The branch length ${\cal L}_\xi( \Tr(\cT_{\xi^{(t)}}))$
is obtained by adding up
all the branch lengths of the trimmed tree  $\Tr(\cT_{\xi^{(t)}})$.
Following the algorithm described in Proposition \ref{algo}, the total branch length
 is given by the sum of the $X_n(\xi^{(t)})$'s
or equivalently 
\beqn
{\cal L}_\xi( \Tr(\cT_{\xi^{(t)}})) 
& = & \sum_{n\geq1} \left( \xi_h^{(t)}(t_{n}^{(t)}) -  \xi_h^{(t)}(s_{n}^{(t)}) \right) \nonumber \\
& = & \sum_{n\geq1} \left( c^h(\xi^{(t)})(t_{n-1}^{(t)}) -  c^h(\xi^{(t)})(t_{n}^{(t)}) \right) \nonumber \\
& = & \lim_{s\uaw\infty} -c^h(\xi^{(t)})(s) \label{ch},
\eeqn
where we wrote $t_{n}^{(t)}:=t_{n}(\xi^{(t)}), s_{n}^{(t)}:=s_{n}(\xi^{(t)})$, and the second line
can be shown as in (\ref{local-x}).

\begin{lemma}
\beqnn
 \lim_{s\uaw\infty} c^h(\xi^{(t)})(s)= c^h(\xi)(t).
\eeqnn
\end{lemma}
\begin{proof}
Since $\xi^{(t)}$ and $\xi$ coincide up to $t$, we have
\be\label{mono}
c^h(\xi^{(t)})(t)=c^h(\xi)(t).
\ee
Furthermore, by Lemma \ref{two-steps-ref},
$$
\forall s\geq t, \ \La_{0,h}(\xi^{(t)})(s) \ = \ \La_{0,h} ( m )(s), \ 
\mbox{where $m(s)=\La_{0,h}(\xi)(t) + 1_{s\geq t}\left(\left(\xi(t) - (s-t)\right)^+ \ - \ \xi(t)\right).$}
$$
The function $m$ is non-increasing on $[t,\infty)$. From this observation, 
we easily get from the definition of the one-sided reflection $\Ga^0(\cdot)$ that 
$$
\forall s\geq0, \ \ \Ga^0(m)(s)\leq \Ga^0(m)(t)\leq h. 
$$
Lemma \ref{pr-s} then
implies that  $\La_{0,h} (m)=\Ga^0(m)$
and that
$c^h(m)=0$ (in other words, no compensator is needed to keep $m$ below level $h$). Finally,
since
$dc^h(m)=dc^h(\xi^{(t)})$ on $[t,\infty)$, 
we have 
$$
\lim_{s\uaw\infty} c^h(\xi^{(t)})(s) \ = \ c^h(\xi^{(t)})(t).
$$
The latter equation combined with (\ref{mono}) completes the proof of the lemma.
\end{proof}

Combining (\ref{eqp}), (\ref{ch}) with the previous lemma,
we get 
\be
\P(\sup_{[0,t]} z^\theta \leq h \ | \ \sigma(w))	\ =  \ \exp\left[2\theta \cdot  c^h(\xi)(t)\right]. \label{last1}
\ee
In order to prove our theorem, it remains to show that
$\La_{0,h}(w)$ is identical in law with a Brownian motion reflected (in a ``standard way'')
on $[0,h]$, and that $-c^h(\xi)$ is the local time of $\La_{0,h}(w)$ at $h$.

\begin{lemma}\label{reflection-inv}
For every continuous function $f$ with $f(0)=0$, $c^h(\Ga^0(f)) = c^h(f)$.
\end{lemma}
\begin{proof}
In Theorem \ref{main1}, we showed that if $g\geq0$ then $g_h\geq0$, 
or equivalently 
$$
\La_{0,h}(g) \leq g.
$$
This implies that for every continuous
non-negative function $g$,
every zero of the function $g$,
is also a zero of the function $\La_{0,h}(g)$. On the other hand,
for any continuous function $f$ with $f(0)=0$,
 the definition of the one-sided Skorohod reflection at $0$ (see (\ref{one-sided}))
implies that
$\Ga^0(f)$ can be written as $f+c$ where $c$ is a non-decreasing continuous
function, only increasing at the zeros of the reflected path $\Ga^0(f)$.
Taking $g=\Ga^0(f)$ in the previous discussion, 
the set of zeros for $\Ga^0(f)$ is included in its $\La_{0,h}(\Ga^0(f))$ counter part,
and 
we get that the compensator $c(t):=-\inf_{[0,t]} f$
(for the one-sided reflection) only increases on the set of zeros of the doubly reflected path 
$\La_{0,h}(\Ga^0(f))$.
Next,
let
$\bar c^{0}$ and $\bar c^h$ be the compensators associated
with the function $\Ga^0(f)$, i.e. $\Lambda_{0,h}(\Gamma^0(f)) = (f + c) + \bar c^{0} + \bar c^h$.
where $\bar c^{0}$ and $\bar c^h$ satisfy the hypothesis of Theorem \ref{Skor} for the function $f+c$.
Since $\bar c^{0}$ only increases at the zeroes of $\La_{0,h}(\Ga^0(f))$,
the functions
$(c+\bar c^{0},\bar c^h)$ must solve the Skorohod equation for $f$ on the interval $[0,h]$.
By uniqueness of the solution, 
this readily implies that $\La_{0,h}(f)=\La_{0,h}(\Ga^0(f))$ and $\bar c^h\equiv c^h(\Ga^0(f))= c^h(f)$. 
\end{proof}

The previous lemma and (\ref{last1}) yield
\beqn
\P(\sup_{[0,t]} z^\theta \leq h \ | \ \sigma(w))	& = & \exp\left[2\theta \cdot  c^h(w)(t)\right]. \label{last}
\eeqn
As already explained in the previous section (see the proof of Lemma \ref{l1}), $\La_{0,h}(w)$ is identical in law with a Brownian motion reflected (in a ``standard way'')
on $[0,h]$, and $-c^h(w)$ is the local time of this process at $h$ (see again the proof of Lemma \ref{l1}
for more details). This completes 
the proof of Theorem \ref{teo1}.

\bigskip

{\bf Acknowledgments.} I thank P. Hosheit for helpful discussions and his careful reading of an early version of the present paper.

\end{document}